\documentclass[12pt]{article}

\usepackage{amsfonts}
\usepackage{amsmath, amsthm}
\usepackage{bbm,mathtools}
\usepackage{hyperref}

\newcommand{\comment}[1]{}

\newcommand{\R}{\mathbb{R}}

\newcommand{\C}{\mathbb{C}}

\newcommand{\DD}{\Delta}
\newcommand{\TD}{\tilde{\Delta}}

\newcommand{\TP}{\tilde{\Phi}}

\newcommand{\Gammatt}{\widetilde{\Gamma}}

\numberwithin{equation}{section}

\newtheorem{lemma}[equation]{Lemma}

\newtheorem{theorem}[equation]{Theorem}

\newtheorem{claim}[equation]{Claim}
\DeclareMathOperator{\Cay}{Cay}
\DeclareMathOperator{\Ch}{Ch}
\DeclareMathOperator{\Aut}{Aut}
\DeclarePairedDelimiter{\norm}{\lVert}{\rVert}

\theoremstyle{definition}
\newtheorem{remark}{Remark}

\newtheorem{definition}[equation]{Definition}

\newtheorem{corollary}[equation]{Corollary}

\begin{document}
\title{Li-Yau inequality on virtually Abelian groups}
\author{Gabor Lippner \and Shuang Liu}
\maketitle

\abstract{We show that Cayley graphs of virtually Abelian groups satisfy a Li-Yau type gradient estimate despite the fact that they do not satisfy any known variant of the curvature-dimension inequality with non-negative curvature. }

%%%%%%%%%%%%%%%%
\section{Introduction}
%%%%%%%%%%%%%%%%

Li and Yau~\cite{LiYau} proved an upper bound on the gradient of positive solutions of the heat equation on manifolds with Ricci curvature bounded from below.  The simplest variant of their result is
\begin{equation}
|\nabla\log u|^2-\partial_t(\log u)=\frac{|\nabla u|^2}{u^2} - \frac{\partial_t u}{u} \leq \frac{n}{2t}, \label{eq:ly}
\end{equation}
 where $u$ is a positive solution of the heat equation \mbox{$(\Delta - \partial_t)u = 0$} on an $n$-dimensional compact manifold with non-negative Ricci curvature. The proof is based on a specific property of such manifolds, the \emph{curvature-dimension inequality} (CD-inequality)
\begin{equation}\label{eq:cd}
\frac{1}{2}\Delta |\nabla f|^2 \geq \langle \nabla f, \nabla \Delta f \rangle + \frac{1}{n}(\Delta f)^2.
\end{equation} 

 It was an important insight by Bakry and Emery \cite{BakEm} that one can use it as a substitute for the lower Ricci curvature bound on spaces where a direct generalization of Ricci curvature is not available. The direct discrete version of the CD-inequality was introduced in~\cite{LinYau}. It is a local notion in the sense that it only depends on 2-step neighborhoods of the nodes of the graph. Its properties were subsequently studied in~\cite{ChLL}, where the authors showed that the discrete CD-inequality implies a weak Harnack-type inequality, but fell short of proving the Li-Yau gradient estimate.
 
In the break-through paper~\cite{BHLLMY} a variant of the CD-inequality was introduced: the so called the exponential curvature-dimension (CDE) inequality. This is still a local notion, its validity depends only on 2-step neighborhoods in the graph. However, for the first time, it was shown that this inequality implies a version of the Li-Yau gradient estimate. 

\begin{theorem}[\cite{BHLLMY}] \label{thm:compactgre} Let $G$ be a finite graph satisfying $CDE(n,0)$, and let $u$ be a positive solution to the heat equation on $G$. Then for all $t>0$
\begin{equation}\label{eq:liyau} \frac{\Gamma(\sqrt{u})}{u} - \frac{\partial_t(\sqrt{u})}{\sqrt{u}} \leq \frac{n}{2t}\ .\end{equation}
\end{theorem}

%%%%%%%%%%%%%%%%
\subsection{Main result}
%%%%%%%%%%%%%%%%

It was shown in~\cite{BHLLMY} that the curvature notion based on CDE-inequality behaves "as expected": complete graphs have positive curvature, lattices have 0 curvature, trees have negative curvature. However, the fact that the CDE-inequality only depends on 2-step neighborhoods leads to an unexpected and undesirable side-effect. The hexagonal lattice, and in more general Cayley-graphs of virtually-Abelian groups (these include periodic planar tilings, among others),  will not satisfy a CDE type inequality with non-negative curvature. This is completely counter-intuitive to the observation that these graphs are essentially flat and hence should ideally have 0 curvature. Our intuition is hence that they should satisfy a Li-Yau type gradient estimate. That is exactly what we show in this paper. For definitions and notation see Section~\ref{sec:notation}.

\begin{theorem}\label{thm:main} Let $\Phi$ be a virtually Abelian group and $\TP \leq \Phi$ a normal Abelian subgroup of index $k < \infty$. Let $S \subset \Phi$ be a finite, symmetric generating set and denote $G = \Cay(\Phi,S)$ the associated Cayley graph. Similarly, let $\tilde{S} \subset \TP$ be a finite, symmetric, conjugation-invariant generating set of $\TP$. There exist constants $K,C > 0$ such that for any solution $w : G \times [0,\infty) \to [0,1]$ of the heat equation on $G$, the shifted solution $u = w + \sqrt{k}$ satisfies
\[ 
\frac{\Gammatt(\sqrt{u})}{K u} - \frac{\partial_t(\sqrt{u})}{\sqrt{u}} \leq \frac{C}{t}\ .
\]
\end{theorem}

The idea of the proof is the following. Let $\tilde{G} = \Cay(\Phi, \tilde{S})$ denote the graph on the full group obtained by generators of the Abelian subgroup. From~\cite{BHLLMY} we know that any positive solution of the heat equation on $\tilde{G}$ satisfies \eqref{eq:liyau} for $n = 2|\tilde{S}|$. We will express solutions of the heat equation on $G$ as a linear combination of solutions of the heat equation on $\tilde{G}$. Then we show that positive linear combinations preserve the validity of \eqref{eq:liyau}. Finally we show that shifting the original solution by a positive constant allows us to turn the original linear combination into a positive linear combination.

 In general, \eqref{eq:ly} is a stronger conclusion than \eqref{eq:liyau}, and it is an easy computation to show that the direct discrete analogue of \eqref{eq:ly} does not hold on graphs.  However, on a manifold, the inequality~\eqref{eq:ly} is equivalent to 
\begin{equation}\label{eq:logly} -\Delta \log u(x,t) \leq n/2t, \end{equation} and  
M\"unch~\cite{Munch} found a new variant of the CD-inequality that implies \eqref{eq:logly} for finite graphs. In particular, he shows that for finite, locally Abelian graphs there exists a constant $n$ depending only on the degree such that \eqref{eq:logly} holds. We shall use this result to prove the following analogue of Theorem~\ref{thm:main}.

\begin{theorem}\label{thm:logmain} 
Under the conditions of Theorem~\ref{thm:main}, we have
\[ -\tilde{\Delta} \log u \leq \frac{C}{t}.\]
\end{theorem} 

%Let $\H$ denote the infinite hexagonal lattice. It can be identified with the Cayley graph of the group $\Phi = \langle a,b,c | a^2 = b^2 = c^2 = abcabc = 1 \rangle$ with the symmetric generating set $S = \{a,b,c\}$ The Cayley graph of the subgroup
%\[ \langle ab, bc, ca \rangle = \TP \leq \Phi \] with symmetric generating set $\TS = \{ab, ba, ac, ca, bc, cb\}$ is the triangular lattice $\TT \subset \H$. The identity element of both groups will be denoted by $1$.

%Let $u(x,t) : \H \times [0,\infty)$ be a solution of the heat equation $\DD u = \partial_t u$ where $\DD$ is the Laplacian on $\H$. The goal of this note is to show that $u$ satisfies a Li-Yau gradient estimate with 0 curvature term, despite that $\H$ does not satisfy any known form of the Curvature-Dimension inequality with 0 curvature.

%The main idea is to relate $u(x,t)$ to solutions $w(x,t) : \TT \times [0,\infty)$ of the heat equation $\TD w = \partial_t w$  where $\TD$ is the Laplacian on $\TT$.

%%%%%%%%%%%%%%%%
\subsection{Notation}\label{sec:notation}
%%%%%%%%%%%%%%%%

For a given locally finite graph $G$ we define the Laplace operator $\Delta = \Delta_G$ acting on a function $f : G \to \R$ as 
\[ \Delta f (x) = \sum_{y \sim x} f(y) - f(x). \]
We will often consider a graph $\tilde{G}$ on the same vertex set at the same time. For convenience we will often abbreviate $\Delta_{\tilde{G}}$ as $\TD$. 

For a given function $g: G \to [0,\infty)$ the heat equation for $u : G \times [0,\infty) \to [0,\infty)$ with initial condition $g$ is the system
\[ \Delta u(x,t) = \partial_t u(x,t), \quad u(x,0) = g(x)\]
where $\Delta$ acts on the first variable of $u$. 

The gradient operator $\Gamma = \Gamma_G$ is defined as 
\[ \Gamma(f)(x) = \sum_{y \sim x} (f(y)-f(x))^2.\]
We will also use the notation $\Gammatt = \Gamma_{\tilde{G}}$.

Throughout the paper we consider a fixed virtually Abelian group $\Phi$ with a finite, symmetric generating set $S$. The unit element will always be denoted by 1. The Cayley graph associated to this generating set will be denoted by $G$. (Edges are given by multiplication by generators on the right.) Since $\Phi$ is virtually Abelian, it has a finite index free Abelian normal subgroup $\TP \leq \Phi$. We fix a generating set $\tilde{S}$ for $\TP$, and denote by $\tilde{G}$ the graph whose vertex set is the same as that of $G$, but the edges are given by multiplication by elements of $\tilde{S}$. Thus, $\tilde{G}$ is a disjoint union of finitely many copies of $\Cay(\TP,\tilde{S})$. 

Every element $x \in \Phi$ defines an automorphism $\phi_x : \TP \to \TP$ by the map $y \mapsto x^{-1} * y * x$. The map $x \mapsto \phi_x$ defines a representation $\Phi \to \Aut(\TP)$ that factors through the quotient $\Phi / \TP$, hence there are at most $k$ different automorphisms obtainable this way and they form a group. 

We will assume that the set $\tilde{S}$ is invariant under the automorphisms $\phi_x$. By the previous remarks such a finite generating set always exists. For example, if $\TP$ is in the center of $\Phi$, any finite generating set can be chosen.

\begin{remark}\label{rem:commutes}
A simple consequence of the invariance of $\tilde{S}$ is that for any $z \in \Phi$ the following two sets are the same:
\begin{equation}\label{eq:set_commute} \{ z*s : s\in \tilde{S}\} = \{s*z : s\in \tilde{S}\}.\end{equation}
Let $f : G \to \C$ be any function, and let $f_z(x) = f(x * z)$. Then \eqref{eq:set_commute} immediately implies the following identities: 
\[ \TD f_z(x) = (\TD f )(x *z),\] and  
\[ \Gammatt(f_z)(x) = \Gammatt(f)(x*z).\]
\end{remark}

%%%%%%%%%%%%%%%%
\section{Relating the heat equation on $G$ and $\tilde{G}$}
%%%%%%%%%%%%%%%%

In this section we explain how to obtain a solution to the heat equation on $G$ as a (possibly infinite) linear combination of solutions to the heat equation on $\tilde{G}$. 

Let $w(x,t) : \tilde{G} \times [0,\infty)$ be the solution of $\TD w = \partial_t w$ with the initial condition $w(x,0) = \delta_{x = 1}$. As a first attempt, let us fix a $\beta > 0$ and try to construct a solution to $\DD u = \partial_t u$ in the form
\begin{equation}\label{eq:u_constr}u(x,t) = \sum_{z \in \Phi} f(z) w(x * z^{-1},\beta t).\end{equation}
Uniform convergence of this sum will be ensured by choosing a nonnegative bounded weight function $f : \Phi \to \R_{\geq 0}$. That is sufficient, since $w$ decays super-exponentially in space according to the Carne-Varopoulos bound~\cite{Carne}. 

Now we can compare $\Delta u$ and $\partial_t u$, and find a sufficient condition on $f$ that ensures $u$ is a solution of the heat equation.

\begin{multline}\label{eq:delta_u} \Delta u (x,t) = \sum_{s \in S} u(x*s)-u(x) = \sum_{z\in \Phi} \sum_{s\in S} f(z) ( w(x*s*z^{-1},\beta t) - w(x * z^{-1},\beta t)) = \\ \sum_{z\in \Phi} f(z)\sum_{s\in S} (w(x*(z*s^{-1})^{-1},\beta t) - w(x* z^{-1},\beta t)) =  \\ \sum_{z\in \Phi} w(x*z^{-1},\beta t) \sum_{s \in S} f(z*s) - f(z) = \sum_{z \in \Phi} (\Delta f)(z) \cdot w(x*z^{-1},\beta t)  \end{multline}

In order to compute the time derivative, we will exploit that $w$ satisfies $\TD w = \partial_t w$, as well as Remark~\ref{rem:commutes}. Let us temporarily denote $w_z(x,t) = w(x*z^{-1},t)$.

\begin{multline} \label{eq:partial_u}
\partial_t u(x,t) = \sum_{z \in \Phi} f(z) \beta \cdot (\partial_t w)(x * z^{-1}, \beta t) = \sum_{z \in \Phi} f(z) \beta \cdot (\TD w)(x * z^{-1}, \beta t) = \\ \beta \sum_{z \in \Phi} f(z) \TD w_z(x,\beta t) = \beta \sum_{z \in \Phi} (\TD f)(z) w_z(x,\beta t) = \beta \sum_{z \in \Phi} (\TD f)(z) w(x*z^{-1},\beta t),
%
%
%
%\sum_{s \in \TS} w(x * z^{-1}*s, \beta t) - w(x*z^{-1},\beta t) =\\  \beta \sum_{z \in \Phi} f(z) \sum_{s \in \TS} w(x * (s^{-1}* z)^{-1}, \beta t) - w(x*z^{-1},\beta t) =  \\ \beta \sum_{z\in \Phi} w(x*z^{-1},\beta  t) \sum_{s \in \TS} f(s*z) - f(z) =\\ 
\end{multline}
Thus combining \eqref{eq:delta_u} and \eqref{eq:partial_u} leads to the following observation.
\begin{lemma}\label{lem:f_beta}
If $\beta \TD f = \Delta f $ then the function $u$ defined in \eqref{eq:u_constr} satisfies $\Delta u = \partial_t u$.
\end{lemma}

The next step is to find a family of functions $f$ that satisfy the conditions of Lemma~\ref{lem:f_beta}. This will be facilitated by the observation that $\Delta$ and $\TD$ commute. Thus we will be able to find functions $f$ that satisfy the condition of Lemma~\ref{lem:f_beta} by constructing joint eigenfunctions of $\Delta$ and $\TD$. 

\begin{claim}\label{claim:commutes} $\Delta\widetilde{\Delta} f =\widetilde{\Delta}\Delta f $ for any function $f$.
\end{claim}

\begin{proof} Writing out the definitions, what we need to check is that
\[ \sum_{s \in S} \sum_{\tilde{s} \in \tilde{S}} f(x * \tilde{s} * s) = \sum_{s \in S} \sum_{\tilde{s} \in \tilde{S}} f(x * s * \tilde{s}),\]
but this follows immediately from Remark~\ref{rem:commutes}.
\end{proof}

%%%%%%%%%%%%%%%%
\section{Constructing periodic solutions}
%%%%%%%%%%%%%%%%

 In this section we will build solutions to the heat equation on $G$ that have almost arbitrary ``periodic'' initial conditions. Fix $n > 1$. Denote  by $\Phi_n$  the finite quotient $\Phi_n = \Phi / \langle s^n : s \in \tilde{S}\rangle$ and by $\TP_n$ the finite Abelian quotient $\TP_n = \TP/ \langle s^n : s \in \tilde{S}\rangle$). The associated graphs will be denoted by $G_n$ and $\tilde{G}_n$ respectively. Let us introduce the quotient map by $\pi_n : \Phi \to \Phi_n$.
 
\begin{definition}
We say that a function $g: \Phi \to \C$ is ($n-$)\textit{periodic} if $g(x) = g(x * s^n)$ for any $s \in \tilde{S}$. Since $\TP$ is normal in $\Phi$, this is equivalent to saying that $g(s^n * x) = g(x)$ for all $s \in \tilde{S}$. It is also equivalent to the existence of a function $h : \Phi_n \to \C$ such that $g = h \circ \pi_n$.
\end{definition}

It is easy to check that the operators $\Delta$ and $\TD$ descend to $G_n$ and that, for either of these operators, a periodic function $g : \Phi \to \C$ is an eigenfunction if and only if it is a lift of an eigenfunction $h : \Phi_n \to \C$.

Let \[\Ch(\TP_n) = \{ \chi : \TP_n \to \C : \chi(x *y) = \chi(x)\chi(y)\}\] denote the set of multiplicative characters on $\TP_n$.
 Fix a multiplicative character $\chi : \in \Ch(\TP_n)$. Next we consider the possible extensions $\chi$ to each coset of $\TP$ as follows. Define the complex vector space 
\begin{equation}\label{eq:inv} V_\chi = \{ g: \Phi_n \to \C  : \forall x \in \Phi_n, \forall s \in \tilde{S}, g(s*x) = \chi(s) g(x) \}.\end{equation}
Clearly for any $g \in V_\chi$ we have, by \eqref{eq:set_commute},
\[ \TD g(x) = \sum_{s \in \tilde{S}} g(x*s) -g(x) = \sum_{s\in \tilde{S}} g(s*x)-g(x) = \lambda_\chi g,\] where $\lambda_\chi = \sum_{s\in \tilde{S}} \chi(s)-1$. It is also clear from the symmetry of $\tilde{S}$ that $\lambda_\chi \in \R_{\leq 0}$, and $\lambda_\chi = 1$ if and only if $\chi \equiv 1$. Let $\mathbbm{1}$ denote the constant 1 character.

We have  $\dim V_\chi =  |\Phi_n : \TP_n| = k$,  and the formula \begin{equation}\label{eq:scalar prod} \langle g,h \rangle = \sum_{x \in \Phi_n} g(x) \overline{h(x)}\end{equation} defines a scalar product on $V_\chi$. 

\begin{claim} $V_\chi$ is an invariant subspace for $\Delta$. 
\end{claim}
\begin{proof}
For any $s \in \tilde{S}$ we have 
\[ \Delta g( s*x) = \sum_{t \in S} g(s*x*t) - g(s*x) = \sum_{t \in S} \chi(s) g(x*t) - \chi(s) g(x) = \chi(s) \Delta g(x).\]
\end{proof}

Thus $\Delta$ is a negative definite, self-adjoint operator on $V_\chi$ so there are $k$ eigenfunctions $f_{\chi,1},\dots, f_{\chi, k} \in V_\chi$ of $\Delta$ with respective eigenvalues $\lambda_{\chi,1}, \dots, \lambda_{\chi,k}$ that form an orthonormal basis with respect to the scalar product \eqref{eq:scalar prod}. 

Since different characters are orthogonal on $\TP_n$, and there are exactly $|\TP_n|$ of them, we get that the system
\[ \{ f_{\chi,j}  : \chi \in \Ch(\TP_n), 1\leq j \leq k\} \] forms an orthonormal basis of $\Delta$-eigenfunctions on $G_n$. One virtue of this particular eigenbasis is that we can bound the supremum norm of its elements.

\begin{lemma}\label{lem:sup_bound}
$\norm{f_{\chi,j}}_{\infty} \leq \sqrt{k/|\Phi_n|}$
\end{lemma}
\begin{proof}
Let's write $f = f_{\chi,j}$ for short. 
We know by construction that $|f|$ is constant along each coset of $\TP_n$. Let these constants be $c_1,\dots, c_k \geq 0$. Then we can write 
\[ 1 = \sum_{x \in \Phi_n} |f(x)|^2 = \sum_{i=1}^k c_i^2 |\TP_n| \geq \max_i(c_i^2) |\TP_n| \]
Thus 
\[ \norm{f}_{\infty} = \max_i(c_i) \leq \sqrt{\frac{1}{|\TP_n|}} = \sqrt{\frac{k}{|\Phi_n|}}.\]
\end{proof}

For each $f_{\chi,j}$ we define 
\[ 0 < \beta_{\chi,j} = \lambda_{\chi,j}/\lambda_\chi\] 
This can be done unless $\chi = \mathbbm{1}$.  The positivity follows from the fact that both $\lambda_\chi$ and $\lambda_{\chi,j}$ are negative. The subspace $V_\mathbbm{1}$ is special, it contains all the functions that are constant along the cosets of $\TP_n$. We may assume that $f_{\mathbbm{1},1}$ is the constant function and define $\beta_{\mathbbm{1},1} = 1$. Thus the following statement holds.

\begin{claim}\label{claim:f_beta}
\[ \beta_{\chi,j} \TD f_{\chi,j} = \Delta f_{\chi,j},\] unless $\chi = \mathbbm{1}$ and $j \geq 2$. Hence the same also holds for $c + f_{\chi,j} \circ \pi_n : G \to \C$ where $c$ is an arbitrary constant, so these functions satisfy the condition Lemma~\ref{lem:f_beta}. 
\end{claim}

Now let $g : G_n \to \C$ be a function that is orthogonal to $f_{\mathbbm{1},j}: j\geq 2$. Then we can express it as a linear combination of all the other eigenfunctions:

\[ g = \sum_{\chi \in \Ch(\TP_n), j = 1,\dots, k} c_{\chi,j} f_{\chi,j},\] where $c_{\mathbbm{1},j} = 0$ if $j \geq 2$.

Let us further choose constants $a_{\chi,j} \in \C$ whose value will be determined later. Let us set
\begin{equation}\label{eq:b_def}
B = \sum_{\chi \in \Ch(\TP_n)} \sum_{j = 1}^k c_{\chi,j} a_{\chi,j}
\end{equation}\label{eq:u_def}
Finally, define $u : G \to \C$ with the formula

\begin{equation}\label{eq:u_def2}  u(x,t) = \sum_{\chi \in \Ch(\TP_n)} \sum_{j = 1}^k \sum_{z \in \Phi}c_{\chi,j}\cdot \left( f_{\chi,j}\circ \pi_n (z) + a_{\chi_j}\right)\cdot  w(x * z^{-1}, \beta_{\chi,j} t) - B \end{equation}

\begin{theorem}\label{thm:periodic}
The function $u$ defined in \eqref{eq:u_def2} satisfies the heat equation on $G$ with initial condition $u(x,0) = g\circ \pi_n(x)$.
\end{theorem}

\begin{proof}
That $\Delta u = \partial_t u$ holds follows from Lemma~\ref{lem:f_beta} combined with Claim~\ref{claim:f_beta}. The only thing we have to show is that $u$ satisfies the initial condition. This is a simple calculation based on the choice of $w(x,0) = \delta_{x=1}$.

\begin{multline*} u(x,0) = \sum_{\chi \in \Ch(\TP_n)} \sum_{j = 1}^k \sum_{z \in \Phi}c_{\chi,j}\cdot \left( f_{\chi,j}\circ \pi_n (z) + a_{\chi_j}\right) \cdot  w(x * z^{-1}, 0)  - B = \\ = \sum_{\chi \in \Ch(\TP_n)} \sum_{j = 1}^k \sum_{z \in \Phi}c_{\chi,j}\cdot \left( f_{\chi,j}\circ \pi_n (z) + a_{\chi_j}\right) \cdot  \delta_{x * z^{-1}=1}- B =\\= \sum_{\chi \in \Ch(\TP_n)} \sum_{j = 1}^k c_{\chi,j}\cdot \left( f_{\chi,j}\circ \pi_n (x) + a_{\chi_j}\right)-B =\\=  \sum_{\chi \in \Ch(\TP_n)} \sum_{j = 1}^k c_{\chi,j}\cdot f_{\chi,j}\circ \pi_n (x)=g(x)
\end{multline*}
\end{proof}

We are interested in non-negative real initial conditions $g(x)$ for which the solution $u(x,t)$ will also be real and non-negative. Then we can take the real part of both sides of \eqref{eq:u_def2} to obtain

\begin{equation}\label{eq:u_real_part}  u(x,t) = \sum_{\chi \in \Ch(\TP_n)} \sum_{j = 1}^k \sum_{z \in \Phi} \Re\left(c_{\chi,j}\cdot \left( f_{\chi,j}\circ \pi_n (z) + a_{\chi_j}\right)\right) \cdot  w(x * z^{-1}, \beta_{\chi,j} t) - \Re(B) 
\end{equation}
The main idea of the proof of Theorem~\ref{thm:main} is to express the solution $u(x,t)$ as a non-negative linear combination of solutions on $\tilde{G}$ for which the gradient estimate is already known. 

Our goal with introducing the constants $a_{\chi,j}$ is to force all coefficients appearing in \eqref{eq:u_real_part}  to be non-negative. This can be done by setting 
\begin{equation}\label{eq:a_choice} a_{\chi,j} = \left\{ \begin{array}{lcr} \frac{|c_{\chi,j}|}{c_{\chi,j}} \norm{f_{\chi,j}}_{\infty} &:& c_{\chi,j} \neq 0 \\ 0 &:& c_{\chi,j} = 0 \end{array} \right. \end{equation}
Unfortunately the use of the constants $a_{\chi,j}$ come at the cost of having to deal with the constant $B$. Next we show that $B$ can be bounded in terms of $g$ but independently of $n$.

\begin{lemma}
If the $a_{\chi,j}$'s are chosen according to \eqref{eq:a_choice} then $B = \sum a_{\chi,j} c_{\chi,j} \leq \sqrt{k} \norm{g}_2$.
\end{lemma}

\begin{proof}
First, it is clear that $B \leq \sum |c_{\chi,j}|\cdot \norm{f_{\chi,j}}_\infty$. Let us recall that $c_{\chi,j} = \langle g, f_{\chi,j} \rangle $ Then, by Cauchy-Schwarz and Lemma~\ref{lem:sup_bound} we get
\[ B \leq \sqrt{  \sum |\langle g, f_{\chi,j} \rangle|^2} \sqrt{\sum \norm{f_{\chi,j}}_\infty^2}  \leq  \norm{g}_2 \sqrt{|\Phi_n| \frac{k}{\Phi_n}} = \sqrt{k} \norm{g}_2.\]
\end{proof}

We can summarize the results of this section as follows.

\begin{corollary}\label{cor:periodic}
Let $g : G_n \to [0,\infty)$ be orthogonal to the span of $\{ f_{\mathbbm{1},j} : 2\leq j \leq k\}$.  Let $u(x,t)$ denote solution of the heat equation on $G$ with initial conditions $u(x,0) = g\circ \pi_n(x)$. Then there exist a non-negative bounded weight function $q(\chi,j,z)$ such that 
\begin{equation}\label{eq:u_pos_comb} u(x,t) + \sqrt{k} \norm{g}_2 = \sum_{\chi, j, z} q(\chi,j,z) w(x*z^{-1},\beta_{\chi,j} t) \end{equation}
\end{corollary}

\begin{remark}\label{rem:ortho} The condition that $g$ be orthogonal to the eigenfunctions $f_{\mathbbm{1},j} : 2\leq j \leq k$ is equivalent to requiring that the sum of $g$ along any coset of $\TP_n$ is the same. This is clear, since the orthogonal projection of $g$ onto $V_\mathbbm{1}$ is obtained by averaging $g$ along each coset, and thus $g$ is orthogonal to $f_{\mathbbm{1},j} : 2 \leq j \leq k$ if and only if this projection coincides with $f_{\mathbbm{1},1}$ -- the constant function.
\end{remark}

%%%%%%%%%%%%%%%%
\section{Gradient estimate}
%%%%%%%%%%%%%%%%

 Let us recall from~\cite{BHLLMY} that since $w(x,t)$ is a solution of the heat equation on $\tilde{G}$, it satisfies the Li-Yau estimate
\[ \frac{\tilde{\Gamma}(\sqrt{w})}{w} - \frac{\partial_t w}{2w} \leq \frac{C}{t} \] with a constant $C$ depending only on $|\tilde{S}|$. 
Then, by Remark~\ref{rem:commutes}, the function $w_{\beta,z}(x,t) = w(x*z^{-1},\beta t)$ satisfies
\[ \frac{\tilde{\Gamma}(\sqrt{w_{\beta,z}})}{w_{\beta,z}} - \frac{\partial_t w_{\beta,z}}{2\beta w_{\beta,z}} \leq \frac{C}{\beta t}, \] or equivalently
\begin{equation}\label{eq:betagradient} \beta \tilde{\Gamma}(\sqrt{w_{\beta,z}}) - \partial_t w_{\beta,z}/2  \leq \frac{C}{t} w_{\beta,z}. \end{equation}

This will be useful if we can establish a global lower bound on the possible $\beta$'s appearing. Our choices for $\beta$ will be the $\beta_{\chi,j}$ family introduced in the previous section.

\begin{lemma}\label{lem:beta_bound}
There is a constant $K$ depending only on $\Phi, S, \tilde{S}$ such that for any fixed $n$ we have $\beta_{\chi,j} \geq 1/K$ for all $\mathbbm{1} \neq \chi \in \Ch(\TP_n)$ and $1 \leq j \leq k$.
\end{lemma}

\begin{proof}
Let's fix $\chi$ and $j$, and let's write $\beta = \beta_{\chi,j}$ and $f = f_{\chi, j}$ for short. 
By Claim~\ref{claim:f_beta} we have $\beta \TD f = \Delta f$. Since $f_{\chi,j}$ is defined on the finite graph $G_n$, we can take scalar product of both sides with $f_{\chi,j}$ and use ``integration by parts'' to get
\begin{equation}\label{eq:edge_sum} \beta_{\chi,j} \sum_{x \in G_n} \sum_{t \in \tilde{S}} (f(x*t)-f(x))^2 = \sum_{x \in G_n} \sum_{s \in S} (f(x*s)-f(x))^2\end{equation}

Each $t \in \tilde{S}$ can be written as a word in $S$. Suppose $t = s_{t,1} s_{t,2} \dots s_{t,r}$ where $s_{t,1},s_{t,2},\dots, s_{t,r} \in S$. (Of course $r = r(t)$ may depend on $t$.) Then we can use Cauchy-Schwarz to obtain
\begin{multline*} 
(f(x*t) - f(x))^2 = \left( \sum_{i=1}^r f(x*s_{t,1}*s_{t,2}* \cdots * s_{t,i}) - f(x*s_{t,1}*s_{t,2}*\cdots * s_{t,i-1})  \right)^2 \leq \\ r  \sum_{i=1}^r \left(f(x*s_{t,1}*s_{t,2}* \cdots * s_{t,i}) - f(x*s_{t,1}*s_{t,2}*\cdots * s_{t,i-1})\right)^2 .
\end{multline*}
Summing this for all $x\in G_n$, for any $i$ the expression $x * s_{t,1} * \cdots * s_{t,i-1}$ also runs exactly over each element of $G_n$. Thus we get 
\[ \sum_{x\in G_n} (f(x*t)-f(x))^2 \leq r \sum_{i=1}^{r(t)}( f(x* s_{t,i})-f(x))^2. \]
Now we sum this last expression for all $t \in \tilde{S}$ to get
\begin{multline*} \sum_{x\in G_n} \sum_{t \in \tilde{S}} (f(x*t)-f(x))^2 \leq r \sum_{x \in G_n} \sum_{t \in \tilde{S}} \sum_{i=1}^{r(t)} (f(x*s_{t,i})- f(x))^2 \leq \\ \leq  r M \sum_{x\in G_n}\sum_{s \in S} (f(x*s)-f(x))^2, \end{multline*}
 where $M$ is the largest of the multiplicities of the elements of the multi-set $\{ s_{t,i} : t\in \tilde{S}, 1 \leq i \leq r(t)\}$. 
 Combining the last estimate with \eqref{eq:edge_sum} we get $\beta \geq 1/(r M)$, so $K = r M$ is a valid choice.
\end{proof}

\begin{lemma}\label{lem:lincomb} $\tilde{\Gamma}\left(\sqrt{\cdot}\right)$ is a convex operator, that is for non-negative weights $c_i$ and non-negative functions $f_i$ we have
\[ \tilde{\Gamma}\left(\sqrt{\sum_i c_i f_i}\right) \leq \sum_i c_i \tilde{\Gamma}\left(\sqrt{f_i}\right). \]
\end{lemma}

\begin{proof}
By the definition of $\tilde{\Gamma}$ all we have to check is that, assuming all sums are convergent and all numbers are non-negative,
\[ \left( \sqrt{\sum_i c_i a_i} - \sqrt{\sum_i c_i b_i}\right)^2 \leq \sum_i c_i \left(\sqrt{a_i} - \sqrt{b_i}\right)^2.\]
This is a simple Cauchy-Schwarz inequality after opening the brackets and canceling as much as possible.
\end{proof}

\begin{corollary}\label{cor:lincomb}For any nonnegative bounded weight function $q(\chi,j,z)$, we have
\[\tilde{\Gamma}\left(\sqrt{\sum_{\chi,j,z}q(\chi,j,z) w_{\beta_{\chi,j},z}}\right) \leq \sum_{\chi,j,z}q(\chi,j,z)\tilde{\Gamma}(\sqrt{w_{\beta_{\chi,j},z}})\]
\end{corollary}

We are ready to prove Theorem~\ref{thm:main}. We will actually show the following slightly stronger statement.

\begin{theorem}
Let $g \in L_2(G)$ a non-negative function. Let $u(x,t)$ denote the solution of the heat equation on $G$ with initial condition $u(x,0) = g(x) + k\norm{g}_2$. Then
\[ 
\frac{\Gammatt(\sqrt{u})}{K u} - \frac{\partial_t(\sqrt{u})}{\sqrt{u}} \leq \frac{C}{t}\ .
\]
\end{theorem}

\begin{proof}
For each $n > 1$, choose a subset $H_n \subset G$ that contains exactly one element from each set $\pi_n^{-1}(x) : x \in \Phi_n$ in such a way that $1 \in H_n$ but that this element 1 is as far from the boundary of $H_n$ as possible with respect to the graph distance. (The set $H_n$ will look like a (skew) ball around the element 1 in $G$.) In particular we have $H_1 \subset H_2 \subset H_3 \dots$, and the distance between 1 and the boundary of $H_n$ clearly goes to infinity as $n$ goes to infinity.

Now we describe a simple way to modify $g$ and  turn it into a function that satisfies the conditions of Corollary~\ref{cor:periodic}. Let $h_n$ denote the ``restriction'' of $g$ to $\Phi_n$ through $H_n$. That is, let $h_n : \Phi_n \to \R$ be the function defined by $h_n(x) = g( \pi_n^{-1}(x) \cap H_n)$. Note that $h_n \circ \pi_n$ coincides with $g$ on $H_n$ and is periodic. 

However $h_n$ still might not be orthogonal to the subspace spanned by $f_{\mathbbm{1},j} : 2 \leq j \leq k$. We can ensure this orthogonality, according to Remark~\ref{rem:ortho}, by modifying $g$ on $H_n \setminus H_{n/2}$ in such a way that for any coset $D$ of $\TP_n$ in $\Phi_n$, the expression $\sum_{x\in D} g( \pi_n^{-1}(x) \cap H_n )$ is independent of $D$. This increases $\norm{g_{|H_n}}_2$ at most $\sqrt{k}$-fold. (The worst case is when $g$ was 0 on all but one of the cosets.) Let $g'_n$ denote the modified function, and let $h'_n$ denote its restriction to $\Phi_n$ as explained in the previous paragraph. 

We get that $\norm{h'_n}_2 \leq \sqrt{k} \norm{g}_2$, and that $h'_n$ is orthogonal to the eigenfunctions $f_{\mathbbm{1},j} : 2 \leq j \leq k$. Finally $h'_n \circ \pi_n$ is a periodic function that coincides with $g$ on $H_{n/2}$. 
Let us denote by $u_n(x,t)$ the solution of the heat equation with initial conditions $u(x,0) = h'_n \circ \pi_n(x)$. Then by Corollary~\ref{cor:periodic} there is a bounded non-negative weight function $q(\chi,j,z)$ such that 
\[ u_n + \sqrt{k} \norm{h'_n}_2 = \sum_{\chi, j ,z} q(\chi,j,z) w_{\beta_{\chi,j}, z}. \] 
By Lemma~\ref{lem:beta_bound} there is a constant $K$ independent of $n$ such that each $\beta_{\chi,j} \geq 1/K$, hence by \eqref{eq:betagradient}, for every $\chi,j$ pair 
\[ \frac{1}{K} \Gammatt(\sqrt{w_{\beta_{\chi,j},z}}) - \partial_t w_{\beta_{\chi,j},z}/2 \leq \frac{C}{t} w_{\beta_{\chi,j},z}. \]
Thus by Corollary~\ref{cor:lincomb} the same holds for $u_n + \sqrt{k} \norm{h'_n}_2$, and thus also for $v_n = u_n + k \norm{g}_2$. So we get that for any point $x \in G$ and any $t \geq 0$
\begin{equation}\label{eq:v} \frac{\Gammatt(\sqrt{v_n})(x,t)}{K v_n(x,t)} - \frac{\partial_t v_n(x,t)}{2 v_n(x,t)} \leq \frac{C}{t}.\end{equation}

Finally, we let $n \to \infty$. The theorem then follows from Claim~\ref{claim:limit} below. 
\end{proof}

\begin{claim}\label{claim:limit} For any $x \in G$ and $ t \geq 0$ we have $\lim_{n\to \infty} u_n(x,t) = u(x,t)$ and $\lim_{n\to \infty} \partial_t u_n(x,t) = \partial_t u(x,t)$.
\end{claim}
\begin{proof}
Let $p_n(x) = g(x) - h'_n \circ \pi_n(x)$, and denote $U_n(x,t)$ the solution of the heat equation with initial conditions $U_n(x,0)=p_n(x)$. Then clearly $u(x,t) = U_n(x,t)+u_n(x,t)$. Note that $\partial_t u, \partial_t u_n$, and $\partial_t U_n$ are also solutions of the heat equation with initial conditions $\Delta g, \Delta g-p_n, \Delta p_n$ respectively. 

For a fixed $x$, if $n$ is large enough, the initial condition function $p_n$ vanishes on a ball of radius $R_n$ around $x$. Since $p_n$ is bounded on the whole $G$, by the Carne-Varopoulos bound the amount of heat that can diffuse from a fixed starting point $y$ to $x$ behaves like $e^{-d(x,y)^2/t}$. Thus 
\[ U_n(x,t) \leq c \norm{p_n}_{\infty} \sum_{y : d(x,y) > R_n} e^{-\frac{d(x,y)^2}{t}}. \]
Since the volume of the balls of radius $R_n$ grow polynomial, this sum clearly decays exponentially to 0 as $R_n \to \infty$. The same holds for $\partial_t U_n$, since $\Delta p_n$ also vanishes on a ball of radius growing to infinity around any particular point $x$ as $n \to \infty$.
\end{proof}

Finally let briefly indicate what modifications are necessary to obtain Theorem~\ref{thm:logmain}.

\begin{proof}[Proof of Theorem~\ref{thm:logmain}]
From~\cite{Munch} we get that there is a constant $C>0$ such that for any solution $w : \tilde{G} \to [0,\infty)$ of the equation $\TD w = \partial_t w$ that is periodic with respect to $\langle s^n : s \in \tilde{S} \rangle$ satisfies 
\begin{equation}\label{eq:logliyau} -\TD \log w \leq \frac{C}{t}.
\end{equation}
The reason $w$ is a priori required to be periodic is because~\cite{Munch} only proves \eqref{eq:logliyau} for finite graphs. However, by the same method as used in the proof of Theorem~\ref{thm:main}, and in particular by Claim~\ref{claim:limit} this implies that \eqref{eq:logliyau} holds for arbitrary $w$. Then, just as in \eqref{eq:betagradient}, we find that $-\beta \log w_{\beta,z} \leq \frac{C}{t}$.  Thus the statement follows, as long as we provide the analog of Lemma~\ref{lem:lincomb} for $-\TD \log$. However $\TD$ is linear, and $\log$ is concave, so $-\TD \log$ is convex, so $-\log \sum c_i f_i \leq - \sum c_i \log f_i$ for real numbers $c_i \geq 0$ and functions $f_i \geq 0$. 
\end{proof}

\noindent
Gabor Lippner,\\
Department of Mathematics, Northeastern University, Boston, Massachusetts\\
\texttt{g.lippner@neu.edu}

\smallskip
\noindent
Shuang Liu,\\ 
Department of Mathematics, Renmin University of China, Beijing, China\\
\texttt{cherrybu@ruc.edu.cn}

\end{document}